\documentclass[letterpaper, 10 pt, conference]{ieeeconf}  

\IEEEoverridecommandlockouts                              

\usepackage{graphicx,bm,amsmath}
\usepackage{amssymb}
\usepackage{epstopdf}
\usepackage{scrextend} 
\usepackage[ruled,norelsize,noend]{algorithm2e}
\usepackage{enumerate}

\usepackage{relsize}

\allowdisplaybreaks

\newcommand{\agent}{\text{MO}}
\newcommand{\lmp}{\lambda}
\newcommand{\blmp}{\bm \lambda}
\newcommand{\bnu}{\bm \nu}
\newcommand{\bomega}{\bm \omega}
\newcommand{\Bomega}{\bm \Omega}

\newcommand{\D}{\mathcal{D}}
\newcommand{\E}{\mathcal{E}}
\newcommand{\Eb}{{\bf E}}
\newcommand{\Ed}{\vec{\mathcal{E}}}
\newcommand{\f}{{\bf f}}
\newcommand{\gbf}{\bar g}
\newcommand{\gtf}{\tilde g}
\newcommand{\G}{\mathcal{G}}

\newcommand{\h}{{\bf h}}

\newcommand{\im}{{\bf i}}

\newcommand{\Lb}{\mathcal{L}}
\newcommand{\M}{\mathcal{M}}
\newcommand{\N}{\mathcal{N}}
\newcommand{\Pb}{{\bf P}}

\newcommand{\Rr}{\mathbb{R}}
\newcommand{\Ss}{\N_\mathcal{S}}
\newcommand{\s}{{\bf s}}
\newcommand{\T}{\mathcal{T}}

\newcommand{\pb}{{\bf p}}

\newcommand{\eb}{{\bf e}}

\newcommand{\rb}{{\bf r}}
\newcommand{\x}{{\bf x}}
\newcommand{\X}{{\bf X}}
\newcommand{\y}{{\bf y}}
\newcommand{\bz}{{\bf z}}
\newcommand{\ub}{{\bf u}}

\newcommand{\cb}{{\bf c}}
\newcommand{\pbl}{\underline{\bf p}}
\newcommand{\pbu}{\bar{\bf p}}
\newcommand{\ebl}{\underline{\bf e}}
\newcommand{\ebu}{\bar{\bf e}}
\newcommand{\pl}{\underline{p}}
\newcommand{\el}{\underline{e}}
\newcommand{\eu}{\bar{e}}
\newcommand{\ipsi}{\mathit{\Psi}}
\newcommand{\pu}{\bar{p}}
\newcommand{\vb}{{\bf v}}

\newcommand{\ze}{{\bm 0}}
\newcommand{\thetal}{\underline{\theta}}
\newcommand{\thetau}{\bar{\theta}}
\newcommand{\btheta}{{\bm \theta}}
\newcommand{\Btheta}{{\bm \Theta}}

\makeatletter
\newcommand{\removelatexerror}{\let\@latex@error\@gobble}
\makeatother

\providecommand{\ex}[1]{\ensuremath{\mathsmaller{\times}}{\scriptsize10\ensuremath{^{\text{-#1}}}}}

\DeclareMathOperator*{\argmax}{\arg\!\max}

\newtheorem{definition}{Definition}
\newtheorem{lemma}{Lemma}
\newtheorem{theorem}{Theorem}
\newtheorem{corollary}{Corollary}

\title{\LARGE \bf
A Decentralized Mechanism for Computing Competitive Equilibria in Deregulated Electricity Markets$^*$
}

\author{Erik~Miehling and Demosthenis~Teneketzis$^{\dagger}$
\thanks{*This work was supported by NSF grant CNS-1238962.}
\thanks{$^{\dagger}$E. Miehling and D. Teneketzis are with the Department 
of Electrical Engineering and Computer Science, University of Michigan, Ann Arbor, MI, USA, 48109. E-mail: miehling@umich.edu.}}

\begin{document}

\maketitle
\thispagestyle{empty}
\pagestyle{empty}

\begin{abstract}
With the increased level of distributed generation and demand response comes the need for associated mechanisms that can perform well in the face of increasingly complex deregulated energy market structures. Using Lagrangian duality theory, we develop a decentralized market mechanism that ensures that, under the guidance of a market operator, self-interested market participants:\! generation companies (GenCos), distribution companies (DistCos), and transmission companies (TransCos), reach a competitive equilibrium. We show that even in the presence of informational asymmetries and nonlinearities (such as power losses and transmission constraints), the resulting competitive equilibrium is Pareto efficient.
\end{abstract}


\section{INTRODUCTION}

From the introduction of the Public Utilities Regulatory Policies Act  (PURPA) in 1978 to the establishment of the Energy Policy Act in 1992, the deregulation of electricity markets in the United States has grown continuously, primarily under the appeal of increased technological competition and innovation. Today, despite cases of market manipulation (such as the California electricity crisis in 2000-2001), many large electricity markets are, at least in some capacity, deregulated. This transition has been centered around the formation of specialized firms for generation, transmission, and distribution, to name a few, with markets typically consisting of the following companies \cite{christie1996load} (termed \emph{market participants}): \emph{generation companies} (GenCos) who produce and sell power, \emph{transmission companies} (TransCos) who own the transmission assets and are responsible for transmitting power across the grid, and \emph{distribution companies} (DistCos) who own the distribution networks and are tasked with buying power from GenCos and distributing it to consumers. 

The primary goal in an energy market is determining an outcome that is not only economically \emph{optimal} (that is, it is \emph{Pareto efficient} \cite{mas1995microeconomic,kirschen2004}) but also satisfies the physical constraints of the system. 
Achieving this goal in a deregulated energy market is complicated by many factors. First, market participants possess informational asymmetries arising from private cost information and localized system knowledge. Additionally, there are power losses within the network, limits on transmission lines, contingency considerations, lack of efficient power storage capabilities, and the underlying, fundamental rule that power flow is dictated by the laws of physics (Kirchhoff's laws). These physical laws result in a phenomenon in power systems termed \emph{loop flow}, creating externalities in the market for electricity \cite{chao1996}.

Centralized market mechanisms are traditionally the approach used for determining the optimal, feasible outcome of the market \cite{stoft2002power}. Under these approaches a centralized market operator receives \emph{bids} from the market participants, in the form of cost/benefit functions and technical constraints, and solves a large-scale centralized optimization problem to determine the market clearing outcome. This outcome consists of a physically feasible operating point as well as a vector of bus-specific power prices termed \emph{locational marginal prices} (LMPs). Unfortunately, centralized mechanisms suffer from some drawbacks. First, reporting cost and technical information raises privacy concerns for market participants. Also, as systems grow in size, the centralized optimization problem can become prohibitively large. This is exacerbated by the recent surge in distributed generation and demand side participation \cite{papadaskalopoulos2013decentralized}, increasing the dimensionality and complexity of the problem further and potentially making centralized mechanisms computationally intractable.

In hopes of avoiding these drawbacks, we introduce a decentralized market mechanism which achieves the economically optimal outcome, honoring the informational asymmetries of the problem and considering important nonlinearities of the system (such as power losses and limits on transmission lines). The electricity market model consists of multiple market participants, DistCos, GenCos, and TransCos, and a single market operator. Our model allows for the consumption centers of each DistCo and the production centers of each GenCo to be distributed across the network. For example, a given GenCo could own generators at multiple buses in the network (a portfolio of plants). Additionally, our model allows for the ownership of transmission lines in the system to be partitioned among multiple TransCos. The market operator is responsible for obtaining a market clearing outcome. The process of achieving this market clearing outcome, termed a \emph{decentralized market mechanism}, is based on principles from Lagrangian duality theory, specifically making use of the \emph{dual decomposition} method \cite{bertsekas1999nonlinear}. The mechanism, which we refer to as the \emph{pricing process}, consists of an iterative price response and price update procedure. All market participants are assumed to act in a self-optimizing manner, that is, given the current LMPs they adjust their decision variables in order to maximize their financial surplus subject to their own local physical and operational constraints. This allows, for instance, for DistCos to exercise flexible demand participation for the elastic component of their total demand and for GenCos to self-dispatch. DistCos and GenCos optimize independently, reporting their surplus-maximizing consumption and production profiles, respectively. TransCos partake in a cooperative message exchange process to reach an operating point that induces power flows that maximize their surpluses for transmitting power along their respective lines. The optimizers are sent to the market operator who is responsible for updating the LMPs in such a way that the self-interested behavior of market participants leads to an outcome that is physically feasible. This outcome, when coupled with the associated set of LMPs, forms a \emph{competitive equilibrium} \cite{mas1995microeconomic,conejo2002}, which we show is Pareto efficient. Under relatively weak conditions (a convex DC approximation and edge-wise positive sums of LMPs), the market participants' optimization problems are convex and the pricing process converges. The pricing process avoids the need for market participants to reveal sensitive information, and additionally, the mechanism scales much more effectively than its centralized counterpart. 

The body of research concerning deregulated electricity markets is vast. We focus on papers most similar to ours, primarily including works that develop decentralized market mechanisms (under perfect competition) using Lagrangian duality techniques. Duality theory allows one to solve the computationally simpler dual problem; however, this can result in a non-zero \emph{duality gap} in general. The authors of \cite{conejo2002,conejo2001} construct a market model consisting of GenCos, DistCos, and a single TransCo, considering a fully nonlinear AC power flow model. Their decentralized mechanism, based on a dual approach, is conjectured to converge to a zero duality gap solution under \emph{profit optimality} and a \emph{convexifying market rule} (a restriction of market participants' behavior). Lavaei and Sojoudi \cite{lavaei2012competitive} consider a competitive energy market setting with GenCos, DistCos, and an ISO under the AC power flow model (using an SDP reformulation). Assuming positive LMPs, the authors are able to show convergence to a zero duality gap solution under the assumption of either: a radial network, or, in the case of a mesh network, the existence of a \emph{phase-shifter} for each network cycle. In the absence of phase-shifters, a zero duality gap can be ensured if loads are allowed to be \emph{over-satisfied} (discarding extra power).  Similar mechanisms have been applied in the context of the unit commitment problem (e.g. \cite{zhuang1988towards}, \cite{bard1988short}, \cite{ongsakul2004unit}) and demand response exchange markets (see \cite{nguyen2012walrasian} and \cite{papadaskalopoulos2013decentralized}).

The contributions of this paper are twofold: {\bf 1)} \emph{Modeling generality}: Our model allows for the ownership of power system assets to be partitioned among the market participants. This allows for each DistCo and GenCo to own multiple units that are distributed across the network (existing literature assumes that each participant owns a single unit \cite{conejo2002,conejo2001,lavaei2012competitive}). Our model also allows for ownership of lines to be partitioned among multiple TransCos (\cite{conejo2002,conejo2001} consider a single TransCo); {\bf 2)} \emph{Convergence to a zero duality gap solution}: Existing models contain nonlinearities that either preclude convergence guarantees (\cite{conejo2002,conejo2001}) or require strong sufficient conditions \cite{lavaei2012competitive}. Our convex power flow approximation allows for power losses to be well-approximated while (under natural conditions) ensuring convergence.

\section{ENERGY MARKET MODEL}
\label{sec:model}

We consider a network of $N$ buses, denoted by the set $\N$, connected by transmission lines, denoted by the undirected edge-set $\E$. Each edge, $\{n,m\}\in\E$, has a line limit $K_{nm}=K_{mn}>0$ and an admittance $Y_{nm} = G_{nm} + \im B_{nm}$ which consists of a conductance $G_{nm}=G_{mn}>0$ and a susceptance $B_{nm}=B_{mn}>0$. We set $K_{nm} = 0$ and $Y_{nm} = 0+\im0$ for $\{n,m\}\not\in\E$. Each bus has an associated voltage angle, denoted by $\theta_n$, with the vector of all angles (termed the operating point) denoted by $\btheta = \{\theta_n\}_{n\in\N}$.

In addition to the market operator ($\agent$), the market model in our paper contains three types of agents (market participants): DistCos, denoted by the set $\D = \{1,\ldots,D\}$; GenCos, denoted by $\G=\{1,\ldots,G\}$; and TransCos, denoted by $\T = \{1,\ldots,T\}$. Each DistCo $i\in\D$ owns \emph{consumption units}, consisting of elastic loads at buses $\N_{\D^e}^i\subseteq\N$ and inelastic loads at buses $\N_{\D^s}^i\subseteq\N$. The elastic and inelastic load profiles of DistCo $i\in\D$ are $\eb^i = \big\{e_n^i\big\}_{n\in\N_{\D^e}^i}$ and $\s^i = \left\{s_n^i\right\}_{n\in\N_{\D^s}^i}$, respectively, where $e_n^i\in[\el_n^i,\eu_n^i]$ is the elastic demand and $s_n^i\ge0$ is the (given) inelastic demand of DistCo $i$'s consumption unit at bus $n$. Each GenCo $i\in\G$ owns \emph{generation units} at buses $\N_\G^i\subseteq\N$. The real power injection profile of GenCo $i\in\G$ is $\pb^i = \big\{p_n^i\big\}_{n\in\N_\G^i}$ where $p_n^i\in[\pl_n^i,\pu_n^i]$ is the injection of GenCo $i$'s generation unit at bus $n$. 
Lastly, each TransCo $i\in\T$ owns a set of transmission lines $\E^i$ with ownership of lines in the system partitioned among TransCos, that is, $\E^1\cup\cdots\cup\,\E^T = \E$ and $\E^i\cap\E^j=\varnothing$, $i\neq j$. Each edge-set $\E^i$ has an associated set of buses $\N_\T^i$ defined as the endpoints of the edges in $\E^i$. The associated voltage angle profile of TransCo $i\in\T$ is $\btheta^i = \big\{\theta_n\big\}_{n\in\N_\T^i}$. A sample network can be seen in Fig. \ref{fig:grid}. 
\begin{figure}[htp]
\vspace{-0.1em}
\begin{center}
\includegraphics[width=0.8\columnwidth]{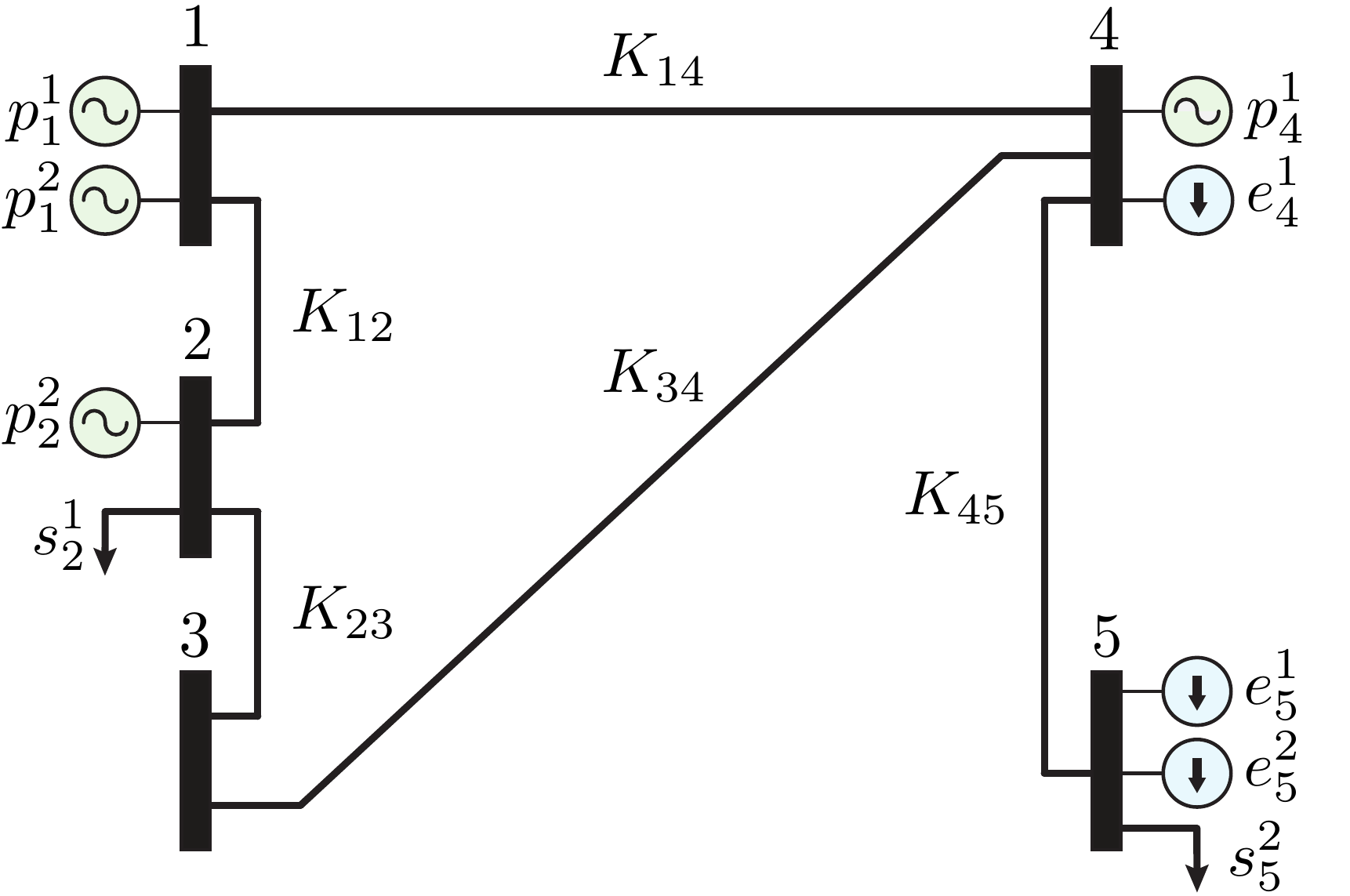}
\caption{\emph{A sample 5-bus network.} GenCo $i=1$ owns generator units at buses $\N_\G^1=\{1,4\}$ corresponding to an injection vector $\pb^1=(p_1^1,p_4^1)$. GenCo $i=2$ has generator units at buses $\N_\G^2=\{1,2\}$, $\pb^2=(p_1^2,p_2^2)$; DistCo $i=1$ has elastic loads at buses $4$ and $5$, $\eb^1 = (e^1_4,e^1_5)$, and an inelastic load at bus $2$, $\s^1=s_2^1$, thus $\N_{\D^e}^1 = \{4,5\}$, $\N_{\D^s}^1 = \{2\}$; and lastly, DistCo $i=2$ has both an elastic and inelastic load at bus $\N_{\D^e}^2 = \N_{\D^s}^2 = \{5\}$, thus $\eb^2 = e^2_5$, $\s^2 = s^2_5$. Bus $3$ is a zero-injection bus. TransCo $i=1$ owns lines $\E^1=\big\{\{1,2\},\{1,4\},\{2,3\}\big\}$ thus $\N_\T^1 = \{1,2,3,4\}$ and TransCo $i=2$ owns lines $\E^2=\big\{\{3,4\},\{4,5\}\big\}$ so $\N_\T^2 = \{3,4,5\}$.}\label{fig:grid}
\end{center}
\vspace{-0.3em}
\end{figure}

For later convenience, we also define $\N_\T^{i,j} := \N_\T^i\cap\N_\T^j$ as the set of shared buses between two TransCos' edge-sets $\E^i$ and $\E^j$ and $\T_n := \{i\in\T\,|\,n\in\N_\T^i\}$ as the set of TransCos that own lines that are connected to bus $n$. 

The load and generation profiles of DistCos and GenCos have associated utilities and costs, respectively. For an elastic load profile $\eb^i$ the aggregate utility (benefit) function of DistCo $i$ is defined as $\ub^i\big(\eb^i\big) := \sum_{n\in\N_{\D^e}^i}u_n^i\big(e_n^i\big)$, where $u_n^i(e_n^i)$ is the benefit associated with elastic demand level $e_n^i$. Similarly, GenCo $i$'s aggregate cost function (total generation cost) is $\cb^i\big(\pb^i\big) := \sum_{n\in\N_\G^i}c_n^i\big(p_n^i\big)$ where $c_n^i(p_n^i)$ represents the cost for producing real power $p_n^i$.

\subsection{Knowledge model}
\label{ssec:know}

We now describe the \emph{knowledge model}, that is, what each of the power system entities knows about the system. Each DistCo $i\in\D$ possesses private information regarding their utility functions $\{u_n^i\}_{n\in\N_{\D^e}^i}$ and any bounds on the elastic load level $\ebl^i=\{\el_n^i\}_{n\in\N_{\D^e}^i},\ebu^i=\{\eu_n^i\}_{n\in\N_{\D^e}^i}$. Each GenCo $i\in\G$ possesses private information regarding their cost functions $\{c_n^i\}_{n\in\N_\G^i}$ and production bounds $\pbl^i=\{\pl_n^i\}_{n\in\N_\G^i},\pbu^i=\{\pu_n^i\}_{n\in\N_\G^i}$. Each TransCo $i\in\T$ knows the connectivity of their region of the network, $(\N_\T^i,\E^i)$, as well as the admittances of the corresponding lines, $Y_{nm}$ for $\{n,m\}\in\E^i$. TransCos also possess private information of the line limits of their transmission lines, $K_{nm}$, $\{n,m\}\in\E^i$. Each DistCo $i\in\D$ knows the inelastic demands at its buses, $\{s_n^i\}_{n\in\N_\D^i}$, whereas the $\agent$ is assumed to know all inelastic demand levels. Furthermore, the $\agent$ knows the location of all DistCo and GenCo units, the network connectivity, and the admittances of all transmission lines in the network.

\subsection{Model Assumptions}
\label{ssec:ass}

We make five core assumptions regarding our model:

\vspace{0.5em}
\noindent  \emph{{\bf Assumption 1} (convex DC approximation)}: We propose a power flow approximation which represents power flow as a convex function of the angle difference. To begin the derivation, recall that by the AC power flow equations \cite{elgerd1973} the real power flowing from bus $n$ to bus $m$ is $P_{nm} = G_{nm}V_n^2 - G_{nm}V_nV_m\cos(\theta_n \!- \theta_m) + B_{nm}V_nV_m\sin(\theta_n \!- \theta_m)$. 
We set voltage magnitudes to 1 p.u., $V_n=1$ $\forall n\in\N$, and assume that voltage angle differences, $\theta_n-\theta_m$, are small (similar to the DC approximation). Using the second-order small angle approximations, $\sin(\theta_n-\theta_m)\approx \theta_n-\theta_m$ and $\cos(\theta_n-\theta_m) \approx 1-\frac{1}{2}(\theta_n-\theta_m)^2$, we can write the expression for the power flow from bus $n$ to bus $m$ as a convex function of the angle difference, $\theta_n-\theta_m$. The resulting approximation, which we term the \emph{convex DC approximation}, dictates that the flow of power on line $(n,m)$ is
\begin{align}\label{eq:flow}
	g(\theta_{nm}) := B_{nm}(\theta_n-\theta_m)+\frac{1}{2}G_{nm}(\theta_n-\theta_m)^2.
\end{align}
The above approximation maintains the asymmetry of the power flow equations, $g(\theta_{nm})\neq -g(\theta_{mn})$, and consequently allows for power losses to be considered (unlike with the DC power flow approximation). The real power losses along line $\{n,m\}$, $L_{nm} = P_{nm} + P_{mn}$, are approximated by $L_{nm} \approx G_{nm}(\theta_n-\theta_m)^2$. For notational convenience, we split Eq. (\ref{eq:flow}) into a DC component, $\gbf(\theta_{nm}) := B_{nm}(\theta_n-\theta_m)$, and a (convex) loss component, $\gtf(\theta_{nm}) := \frac{1}{2}G_{nm}(\theta_n - \theta_m)^2$.

\vspace{0.5em}
\noindent \emph{{\bf Assumption 2} (slack buses)}: Denote the set of slack buses by $\Ss$. We require that each TransCo has exactly one slack bus, that is, $\N_\T^i\cap\Ss$ contains one element for all $i\in\T$. Slack buses serve solely as angle references, that is, $\theta_n = 0$ for all $n\in\Ss$.

\vspace{0.5em}
\noindent \emph{{\bf Assumption 3} (strong convexity)}: We require that all DistCo utility functions $u_n^i$ are strongly concave and all GenCo cost functions $c_n^i$ are strongly convex (this condition is equivalent to strict convexity if the functions are quadratic). See \cite{stott1987} for justification of the convexity assumption.

\vspace{0.5em}
\noindent \emph{{\bf Assumption 4} (positive edge-wise sums of prices)}: We require that all edge-wise sums of locational marginal prices are positive. That is, $\lambda_n+\lambda_m>0$ for all $\{n,m\}\in\E$.\footnote{Note: We are not enforcing this as a constraint in our problem, rather we are only considering topologies where this assumption is naturally satisfied.} Note that this allows $\lambda_n<0$ for some $n$.

\vspace{0.5em}
\noindent \emph{{\bf Assumption 5} (price-taking behavior)}: We assume that the agents (market participants) are price-taking, that is, they assume that the price will remain unchanged if they change their response. This assumption requires that no single agent is large enough (apart from the $\agent$) to influence the price; an assumption which is reasonable in our market model as the number of agents in the system increases.

\section{MAXIMIZING SOCIAL WELFARE}
\label{sec:social}

We are interested in determining the set of variables, consisting of DistCo elastic demand levels $\{\eb^i\}_{i\in\D}$, GenCo real power injection levels $\{\pb^i\}_{i\in\G}$, and an operating point $\btheta$, such that the \emph{social welfare} is maximized, subject to physical and operational constraints. It is known from microeconomic theory that maximizing the social welfare results in a Pareto efficient outcome \cite{mas1995microeconomic}. The single time-period problem can be formally stated as Problem (\ref{prob:central}) below.
\begin{align}
	&\hspace{-2.2em}\max_{\x = (\{\eb^i\}_{i\in\D},\{\pb^i\}_{i\in\G},\btheta)} \, J(\x):=\sum_{i\in\D}\ub^i\big(\eb^i\big)-\sum_{i\in\G}\cb^i\big(\pb^i\big)\label{prob:central}\tag{P}\\
	 \text{s.t. } &\hspace{0.1em} \pb - (\eb + \s) = \f\big(\btheta\big) \label{pc1}\tag{P.i}\\
	 &\hspace{0.1em}\pbl^i \le \pb^i \le \pbu^i,i\in\G\label{pc2}\tag{P.ii}\\
	 &\hspace{0.1em}\ebl^i\le\eb^i\le\ebu^i,i\in\D\label{pc3}\tag{P.iii}\\
	  &\hspace{0.1em}g\big(\theta_{nm}\big) \le K_{nm},g\big(\theta_{mn}\big) \le K_{mn},\{n,m\}\in\E\label{pc4}\tag{P.iv}\\
	 &\hspace{0.1em}\thetal_{nm}\le\theta_{nm}\le\thetau_{nm},\{n,m\}\in\E\label{pc5}\tag{P.v}\\
 	 &\hspace{0.1em}\theta_n=0,n\in\Ss\label{pc6}\tag{P.vi}\\
 	 &\hspace{0.1em}\theta_n\in[-\pi,\pi],n\in\N\label{pc7}\tag{P.vii}
\end{align}
The objective function of Problem (\ref{prob:central}), $J(\x)$, represents the \emph{social welfare} and is defined as the total utility to DistCos minus the total cost to GenCos.

The constraints of problem (\ref{prob:central}) arise from both physical laws and the operational requirements of the power system and the agents. The first constraint (\ref{pc1}), termed the \emph{power balance equation}, takes the form
\begin{align}\label{eq:pb}
	\pb - (\eb + \s) = \f\big(\btheta\big)
\end{align}
where $\pb = (p_1,p_2,\ldots,p_N)$, with $p_n = \sum_{i\in\G}p_n^i$, is the net generation vector and $(\eb+\s)$ is the net demand vector consisting of two components, the elastic demand vector $\eb = (e_1,e_2,\ldots,e_N)$ and the (fixed) inelastic demand vector $\s=(s_1,s_2,\ldots,s_N)$ (with $e_n = \sum_{i\in\D}e_n^i$ and $s_n=\sum_{i\in\D}s_n^i$). The vector $\f\big(\btheta\big) = \big(f_1\big(\btheta\big),f_2\big(\btheta\big),\ldots,f_N\big(\btheta\big)\big)$ denotes the power injections induced by the operating point $\btheta$, where the injection at bus $n$ is defined by the convex function $f_n\big(\btheta\big) = \sum_{m\in\N}g\big(\theta_{nm}\big)$, where $g(\theta_{nm})$ represents the power flow from bus $n$ to $m$ defined in Eq. (\ref{eq:flow}); notice that $g(\theta_{nm})$ is zero if $\{n,m\}\not\in\E$. Constraint (\ref{pc1}) simply states that the injections due to physical laws, $\f(\btheta)$, must agree with the net generation and demand at every bus. Constraints (\ref{pc2}) and (\ref{pc3}) reflect the fact that GenCos/DistCos have bounds on the amount of power they are able to produce/consume. Transmission constraints on the amount of power flowing on each line, constraint (\ref{pc4}), stability constraints on the voltage angle difference, (\ref{pc5}), and slack references, (\ref{pc6}), are also imposed. The last constraint, (\ref{pc7}), is a technical condition that ensures that the voltage angles are well-defined. We group constraints (\ref{pc2})-(vii) into a set denoted by $\X$. It is clear that $\X$ is convex since it is the intersection of half-spaces and convex inequality constraints. Lastly, we assume that Problem (\ref{prob:central}) is feasible.

There are some fundamental difficulties in obtaining a solution to Problem (\ref{prob:central}). First, the problem is nonconvex due the presence of the nonlinear power balance equation. Furthermore, by the discussion in Section \ref{ssec:know}, no single entity in the system has the information required to obtain a solution to Problem (\ref{prob:central}). The remainder of the paper will focus on obtaining a solution to Problem (\ref{prob:central}). 

\section{SURPLUSES \& COMPETITIVE EQUILIBRIA}
\label{sec:surplus}

The notion of a \emph{competitive equilibrium} will be of central importance in obtaining a solution to Problem (\ref{prob:central}). Before we formally define a competitive equilibrium in the context of our problem, we need to discuss some aspects related to the Lagrangian dual function of Problem (\ref{prob:central}).

A partial Lagrangian of Problem (\ref{prob:central}) is formed by dualizing the power balance equation through the vector of dual variables, $\blmp$, where each component $\lmp_n$ represents the locational marginal price of power at bus $n$. Denoting the vector of variables by $\x = (\{\eb^i\}_{i\in\D},\{\pb^i\}_{i\in\G},\btheta)$, and defining $\h(\x):=\f\big(\btheta\big) - \pb +\eb + \s$, the Lagrangian is 
\begin{align}
	\nonumber\!\! \Lb\big(\x,\blmp\big) &:= J(\x) - \blmp^\top\h(\x)\\ 
	\nonumber&\hspace{-2em}=\!\sum_{i\in\D}\!\ub^i\big(\eb^i\big)\!-\!\sum_{i\in\G}\!\cb^i\big(\pb^i\big)\!-\blmp^\top\!\!\left(\f\big(\btheta\big) - \pb +\eb + \s\right)\\
	\nonumber &\hspace{-2em}=\sum_{i\in\D}\sum_{n\in\N_{\D^e}^i}u_n^i\big(e_n^i\big) - \sum_{i\in\G}\sum_{n\in\N_\G^i}c_n^i\big(p_n^i\big)\\
	\nonumber &\hspace{-1em}-\sum_{n\in\N}\lmp_n\left(f_n(\btheta) - \sum_{i\in\G}p_n^i + \sum_{i\in\D}\left(e_n^i + s_n^i\right)\right)\\
	\nonumber &\hspace{-2em}=\sum_{i\in\D}\left(\sum_{n\in\N_{\D^e}^i}\!\!\left[u_n^i\big(e_n^i\big)-\lmp_ne_n^i\right] - \!\!\sum_{n\in\N_{\D^s}^i}\!\!\!\lambda_ns_n^i\right)\\
		&\hspace{-2em}+\sum_{i\in\G}\left(\sum_{n\in\N_\G^i}\left[\lambda_np_n^i - c_n^i\big(p_n^i\big)\right]\right)-\sum_{n\in\N}\lmp_nf_n(\btheta).\label{eq:lagrangian}
\end{align}
Due to the structure of the Lagrangian, Eq. (\ref{eq:lagrangian}), the evaluation the dual function, defined as $\phi(\blmp) = \max_{\x\in\X}\big\{\Lb\big(\x,\blmp\big)\big\}$, is greatly simplified via separable optimizations.
\begin{align}
	\nonumber \phi(\blmp) 	&= \max_{\x\in\X}\big\{\Lb\big(\x,\blmp\big)\big\}\\
			\nonumber &= \sum_{i\in\D}\max_{\eb^i\in\Eb^i}\bigg\{\sum_{n\in\N_{\D^e}^i}\!\!\left[u_n^i\big(e_n^i\big)-\lmp_ne_n^i\right] - \!\!\sum_{n\in\N_{\D^s}^i}\!\!\!\lambda_ns_n^i\bigg\}\\
			\nonumber&\hspace{1.15em}+\sum_{i\in\G}\max_{\pb^i\in\Pb^i}\bigg\{\sum_{n\in\N_\G^i}\left[\lambda_np_n^i - c_n^i\big(p_n^i\big)\right]\bigg\}\\
			&\hspace{1.15em}+ \max_{\btheta\in\Btheta}\bigg\{- \sum_{n\in\N}\lmp_nf_n(\btheta)\bigg\}\label{eq:dualfcn}
\end{align}
where the constraint sets are $\Eb^i = \{\eb^i|\ebl^i\le\eb^i\le\ebu^i\}$, $\Pb^i = \{\pb^i|\pbl^i\le\pb^i\le\pbu^i\}$, and 
$\Btheta = \{(\btheta^1,\ldots,\btheta^T)\in\Btheta^1\times\cdots\times\Btheta^T: \theta_n^i=\theta_n^j,n\in\N_\T^{i,j},i,j\in\T\}$ with each TransCo's feasible set defined as
\begin{align*}
	\Btheta^i := \big\{\btheta^i\big|& g\big(\theta_{nm}\big) \le K_{nm},g\big(\theta_{mn}\big) \le K_{mn},\, \{n,m\}\in\E^i;\\
	 \nonumber \quad& \thetal_{nm}\le\theta_{nm}\le\thetau_{nm},\{n,m\}\in\E^i;\\
	 \nonumber \quad& \theta_n=0,n\in\N_\T^i\cap\Ss;\theta_n\in[-\pi,\pi],n\in\N_\T^i\big\}.
\end{align*}
For later reference, the dual problem of Problem (\ref{prob:central}) is simply
\begin{align}\label{prob:dual}
	\min_{\blmp}\phi(\blmp).\tag{D}
\end{align}

\subsection{Agent Surplus Functions}

The arguments of the maximizations in Eq. (\ref{eq:dualfcn}) represent surplus functions of the agents. This follows from the fact that the dual variables, $\blmp$, of the power balance equation represent locational marginal prices. 
The surplus for DistCo $i\in\D$ for a given demand profile $(\eb^i,\s^i)$ at price $\blmp$ is equal to the utility obtained from $\eb^i$ minus the cost of total demand (sum of elastic and inelastic demand), defined as
\begin{align*}
	\ipsi_\D^i(\eb^i,\blmp) := \sum_{n\in\N_{\D^e}^i}\!\!\left[u_n^i\big(e_n^i\big)-\lmp_ne_n^i\right] - \!\!\sum_{n\in\N_{\D^s}^i}\!\!\!\lambda_ns_n^i.
\end{align*}
The surplus of each GenCo $i\in\G$ is equal to the payment it receives for producing power minus the generation cost,
\begin{align*}
	\ipsi_\G^i(\pb^i,\blmp) := \sum_{n\in\N_\G^i}\left[\lambda_np_n^i - c_n^i\big(p_n^i\big)\right].
\end{align*}
TransCos receive a surplus for facilitating power flow across the network. Congestion and losses in transmission lines creates different valuations for power across the network (represented by LMPs) and results in a discrepancy between the payments received from DistCos and the payments made to GenCos. This creates a surplus (possibly negative) for transmitting power from GenCos to DistCos, termed the \emph{merchandizing surplus}. Under the convex DC approximation (assumption 1), the total merchandizing surplus (argument of the last maximization term in Eq. (\ref{eq:dualfcn})) can be shown to be
\begin{align}
	\nonumber \hspace{-0.78em}\ipsi_\T(\btheta,\blmp) &= - \sum_{n\in\N}\lmp_nf_n(\btheta)= -\!\!\!\!\sum_{(n,m)\in\Ed}\lmp_ng(\theta_{nm})\\
	&\hspace{-3.75em}= \frac{1}{2}\!\!\sum_{(n,m)\in\Ed}\!\!\!\big((\lambda_m-\lambda_n)\bar g(\theta_{nm}) - (\lambda_n + \lambda_m)\tilde g(\theta_{nm})\big)\!\!\!\label{eq:totalms}
\end{align}
where $\Ed$ is the directed edge-set and is defined as the set that contains the pair $(n,m)$ and $(m,n)$ for every edge $\{n,m\}\in\E$. The quantity $\big(\lambda_m-\lambda_n\big)\bar g\big(\theta_{nm}\big) - \big(\lambda_n + \lambda_m\big)\tilde g\big(\theta_{nm}\big)$ is the merchandizing surplus for enabling flow between buses $n$ and $m$ at the price vector $\blmp$. Notice that the first term, $\big(\lambda_m-\lambda_n\big)\bar g\big(\theta_{nm}\big)$, is the familiar expression for the merchandizing surplus under the DC approximation \cite{wu1996folk}. The second term, ${-\big(\lambda_n + \lambda_m\big)\tilde g\big(\theta_{nm}\big)}$, arises from the fact that we are considering losses in our model (see assump. 1).

Since the ownership of transmission lines is partitioned among TransCos we can separate the total merchandizing surplus into each TransCo's merchandizing surplus as 
\begin{align}
	\nonumber \ipsi_\T^i(\btheta^i,\blmp) &= \frac{1}{2}\sum_{(n,m)\in\Ed^i}\!\!\big((\lambda_m-\lambda_n)\bar g(\theta_{nm})\\
	&\hspace{8em} - (\lambda_n + \lambda_m)\tilde g(\theta_{nm})\big)\label{eq:ms}.
\end{align}
We show that, via a message exchange process described in Section \ref{ssec:transco_opt}, TransCos communicate to obtain the angle profile which maximizes $\ipsi_\T(\btheta,\blmp)$ over $\btheta\in\Btheta$ at price $\blmp$.

\subsection{Competitive Equilibria in Energy Markets}

We can now define the concept of a competitive equilibrium in the context of our energy market model. The definition builds upon the one found in \cite{conejo2002}.

\vspace{0.5em}
\begin{definition}[Competitive Equilibrium]\label{def:ce}
A \emph{competitive equilibrium} is defined as the tuple $(\{\hat \eb^i\}_{i\in\D},\{\hat \pb^i\}_{i\in\G},\hat\btheta,\hat\blmp)$ such that\vspace{0.5em}
\begin{enumerate}[(i)]\setlength\itemsep{0.5em}
\item $\left|\left|f_n(\hat\btheta) - \hat p_n + \hat e_n+s_n\right|\right|< \varepsilon$ for all $n\in\N$, $\varepsilon>0$
\item
\begin{itemize}\setlength\itemsep{0.5em}
\item[] \hspace{-2em}$\hat \eb^i(\hat\blmp)$ maximizes $\ipsi_\D^i(\eb^i,\hat\blmp)$ s.t. $\eb^i\in\Eb^i$, $\forall \,i\in\D$
\item[] \hspace{-2em}$\hat \pb^i(\hat\blmp)$ maximizes $\ipsi_\G^i(\pb^i,\hat\blmp)$ s.t. $\pb^i\in\Pb^i$, $\forall \,i\in\G$
\item[] \hspace{-2em}$\hat \btheta^i(\hat\blmp)$ maximizes $\ipsi_\T^i(\btheta^i,\hat\blmp)$ s.t. $\btheta^i\in\Btheta^i$, $\forall \,i\in\T$
\end{itemize}
\end{enumerate}
\end{definition}
\vspace{0.5em}
The above definition states that a competitive equilibrium must not only satisfy the power balance equation (condition (i)) but also result in maximum surplus for all DistCos, GenCos, and TransCos (condition (ii)).

\section{SOLUTION METHODOLOGY}
\label{sec:solution}

Throughout the remainder of the paper we describe a procedure in which the $\agent$ and agents interact in order to obtain a \emph{globally optimal solution} $\x^*$ to the nonconvex primal problem (\ref{prob:central}). The procedure is based on the dual decomposition method; an iterative method that first involves the evaluation of the dual function for a given set of dual variables (prices), followed by an update of the dual variables. 

In the context of the electricity market model in this paper, the evaluation of the dual function is performed in a distributed fashion by the agents. In fact, maximization of surpluses by the agents corresponds exactly to the evaluation of the dual function. DistCos and GenCos maximize in parallel to obtain the optimal profiles for the current price $\blmp^t$, denoted by $\{\eb^i(\blmp^t)\}_{i\in\D}$ and $\{\pb^i(\blmp^t)\}_{i\in\G}$, respectively. TransCos partake in a message exchange process (due to coupling of merchandizing surplus functions) in order to obtain the operating point which maximizes the total merchandizing surplus at the current price, denoted by $\btheta(\blmp^t)$. The $\agent$ uses these maximizers to update the price in such a way as to enforce feasibility (condition (i) of Def. \ref{def:ce}). A block diagram outlining the method can be seen in Fig. \ref{fig:block}. 

\begin{figure}[h!]
\begin{center}
\includegraphics[width=0.95\columnwidth]{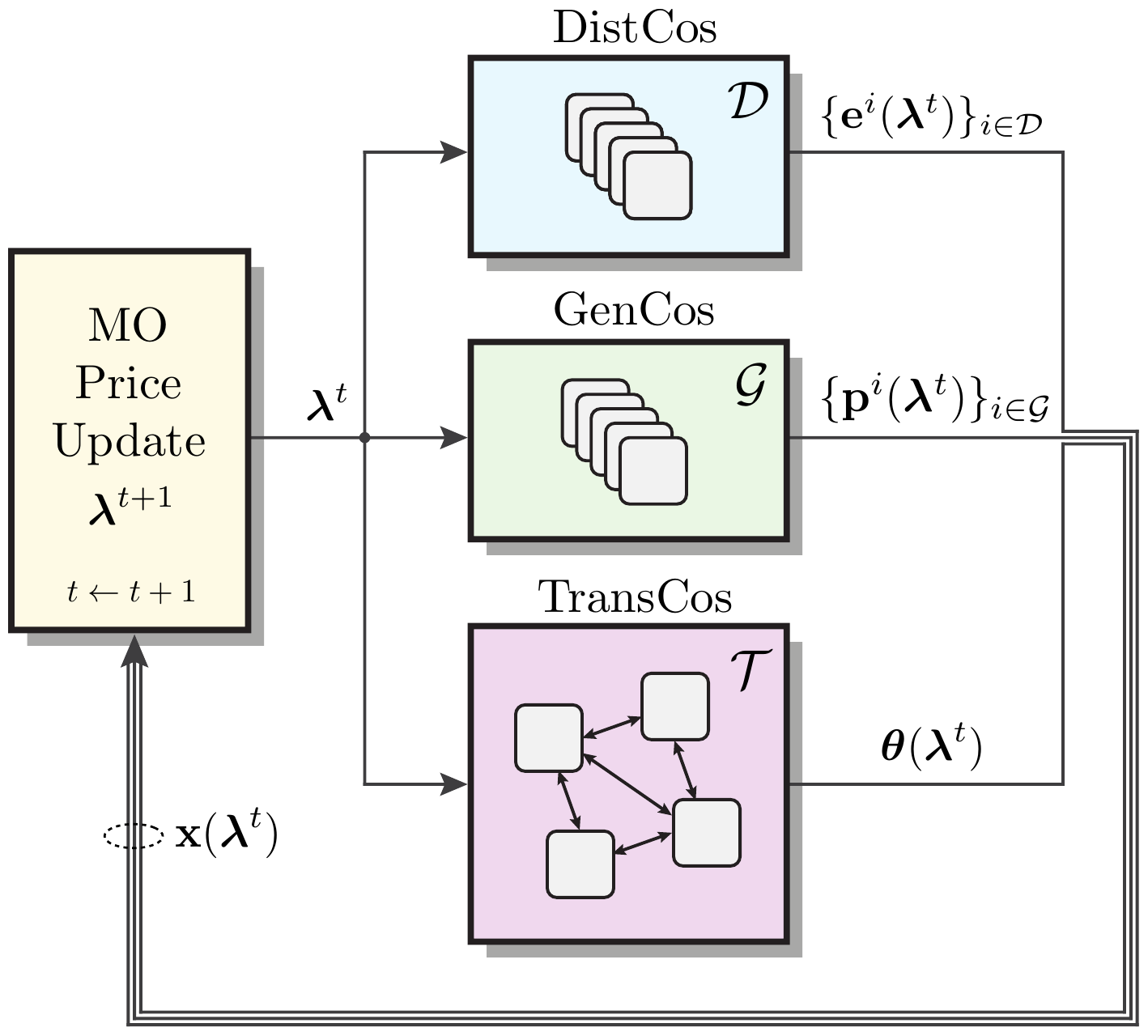}
\caption{\emph{The pricing process:} Given the current price vector $\blmp^t$, DistCo's and GenCo's update the respective components of the consumption profiles $\{\eb^i(\blmp^t)\}_{i\in\D}$ and generation profile $\{\pb^i(\blmp^t)\}_{i\in\G}$, in parallel. TransCo's participate in a message exchange process to reach an angle profile agreement $\btheta(\blmp^t)$. The $\agent$ then updates the price to $\blmp^{t+1}$ using the responses $\x(\blmp^t) = (\{\eb^i(\blmp^t)\}_{i\in\D},\{\pb^i(\blmp^t)\}_{i\in\G},\btheta(\blmp^t))$ (outlined in Section \ref{ssec:update}).
}\label{fig:block}
\end{center}
\vspace{-1em}
\end{figure}
 
\subsection{Price Response - Dual Function Evaluation}
\label{ssec:response}

The first step of the iterative process involves evaluation of the dual function for a given current price vector $\blmp^t$. This is achieved via the following agent surplus maximizations.

\vspace{0.75em}
\subsubsection{DistCo Optimizations} Each DistCo, $i\in\D$, wishes to specify the elastic demand level $\eb^i$ in order to maximize its surplus from buying power (both elastic and inelastic) at the current price $\blmp^t$. Each DistCo $i\in\D$ solves
\begin{align}
	\eb^i(\blmp^t)=\argmax_{\eb^i\in \Eb^i} \ipsi_\D^i(\eb^i,\blmp^t)\label{prob:distco}\tag{$P_\D^i$}
\end{align}
By assumption 3, each $u_n^i$ is strictly concave and therefore the maximizer $\eb^i(\blmp^t)$ of Problem (\ref{prob:distco}) is unique for each $i$.

\vspace{0.75em}
\subsubsection{GenCo Optimizations} Each GenCo, $i\in\G$, wishes to specify the injection levels $\pb^i$ in order to maximize its surplus from selling power at $\blmp^t$. Each GenCo $i\in\G$ solves
\begin{align}
	\pb^i(\blmp^t) = \argmax_{\pb^i\in \Pb^i} \ipsi_\G^i(\pb^i,\blmp^t)\label{prob:genco}\tag{$P_\G^i$}
\end{align}
Again, by assumption 3, the maximizer $\pb^i(\blmp^t)$ of Problem (\ref{prob:genco}) is unique for each $i\in\G$.

\vspace{0.5em}
\subsubsection{TransCo Optimizations}
\label{ssec:transco_opt}
 
Each TransCo, $i\in\T$, aims to specify their voltage angle profile $\btheta^i\in\Btheta^i$ such that the induced flows maximize their merchandizing surplus at the current price, $\ipsi^i_\T(\btheta^i,\blmp^t)$. Doing so is complicated by the fact that there exist buses that are shared between one or more TransCos, that is, $\N_\T^{i,j}\neq\varnothing$ for neighboring $i,j\in\T$. The presence of these shared buses creates coupling between the merchandizing surplus functions of distinct TransCos.

As a result, all neighboring TransCos must negotiate the angle value of their shared buses. Arriving at a system-wide agreement for the shared buses, with each TransCo maximizing their own merchandizing surplus, results in a maximization of the total merchandizing surplus (achieving the value of the last term in Eq. (\ref{eq:dualfcn})). The angle profile agreement is achieved via a message exchange process that is based on the ADMM algorithm \cite{boyd2011distributed} in which neighboring TransCos iteratively exchange the voltage angle values of their shared buses. 

To make use of the ADMM algorithm, it is necessary to write the problem of maximizing the total merchandizing surplus, $\max_{\btheta\in\Btheta}\ipsi_\T(\btheta,\blmp)$, as the equivalent problem
\begin{align}
	\max_{\{\btheta^i\}_{i\in\T},\bz}&\quad \ipsi_\T(\btheta,\blmp^t) = \sum_{i\in\T}\ipsi_\T^i(\btheta^i,\blmp^t)\label{prob:transco}\tag{$P_{\T}$}\\
	 \nonumber \text{subject to } \quad & \btheta^i\in\Btheta^i,i\in\T\\
	 \nonumber & \btheta^i - \bz^i = \ze,i\in\T
\end{align}
where $\bz\in\Rr^N$ is the global variable representing the system-wide angle profile $\btheta$ and $\bz^i = \{z_n\}_{n\in\N_\T^i}$ is the relevant component of $\bz$ corresponding to TransCo $i$'s angle profile. We associate a set of dual variables, $\y^i = \{y_n\}_{n\in\N_\T^i}$, with each of the \emph{consensus constraints} $\btheta^i - \bz^i = \ze$, $i\in\T$. The consensus constraints enforce the angle profiles of TransCos to agree. Defining primal and dual residual norms \cite{boyd2011distributed} as
\begin{align}
r^{(k)} &:= \big(\btheta^{1,{(k)}} - \bz^{1,{(k)}},\ldots,\btheta^{T,{(k)}} - \bz^{T,{(k)}}\big)\label{eq:primalres}\\
s^{(k)} &:= -\rho\big(\bz^{1,{(k)}}-\bz^{1,{(k-1)}},\ldots,\bz^{T,{(k)}}-\bz^{T,{(k-1)}}\big)\label{eq:dualres}
\end{align}
the TransCo message exchange process is given by Alg. 1.
\begin{figure}[!b]
\vspace{-1.5em}
\small
 \removelatexerror
  \begin{algorithm}[H]
   \caption{TransCo Message Exchange Process}
   Initialize $k=0$, $\y^{i,{(0)}}=\ze$ for all $i\in\T$, $\bz^{(0)}=\ze$, $\rho>0$\;
   \While(){$\neg(||r^{(k)}||_2<\varepsilon_{\text{primal}}$ and $||s^{(k)}||_2<\varepsilon_{\text{dual}})$}
   {
       \For( \emph{(parallel optimization and broadcast)}){$i\in\T$}
       {
       	TransCo $i$ solves:
       	\begin{align*}
    		\btheta^{i,{(k+1)}}(\blmp^t) = \argmax_{\btheta^i\in\Btheta^i}\bigg\{&\ipsi_{\T}^i(\btheta^i,\blmp^t)\\ 
    		&\hspace{-8em}-\! \big(\y^{i,{(k)}}\big)^{\!\top}\!(\btheta^i - \bz^{i,{(k)}}) \!-\! \frac{\rho}{2}\left|\left|\btheta^i - \bz^{i,{(k)}}\right|\right|^2\!\bigg\}
    	\end{align*}
          	Broadcast $\{\theta_n^{i,{(k+1)}}\}_{n\in\N_\T^{i,j}}$ to neighboring $j\in\T$;
       }
       \For(\emph{(parallel average and dual variable update)}){$i\in\T$}
       { 	
       		\vspace{0.25em}
                	Average: $z_n^{{(k+1)}} = \displaystyle\frac{1}{|\T_n|}\sum_{j\in\T_n}\theta_n^{j,{(k+1)}}$, $\forall n\in\N_\T^i$\\
		\vspace{0.25em}
       	    	Update: $\y^{i,{(k+1)}} = \y^{i,{(k)}} + \rho(\btheta^{i,{(k+1)}}- \bz^{i,{(k+1)}})$
	}
	Update residuals: set $r^{(k+1)}$, $s^{(k+1)}$ via Eq.'s (\ref{eq:primalres}), (\ref{eq:dualres})\;
   	Update counter: $k\leftarrow k+1$\;
   }
  \end{algorithm}
\end{figure}

By assumption 2, each TransCo $i$ has a slack bus in its set of buses $\N_\T^i$, and the following lemma holds.
\vspace{0.5em}
\begin{lemma}\label{lem:ms}
The merchandizing surplus function $\ipsi_\T^i(\btheta^i,\blmp)$ is strongly concave in $\btheta^i$ for all $i\in\T$.
\end{lemma}
\hspace{1.5em}\emph{Proof:} 
See appendix A.\hfill$\QED$
\vspace{0.5em}

\noindent Lemma \ref{lem:ms} and the convergence result of the ADMM in \cite{boyd2011distributed} (p.17) lead to the following corollary.
\vspace{0.5em}
\begin{corollary}\label{cor:admm}
Alg. 1 generates iterates $\{\bz^{(k)}\}$ that converge to the unique solution $\btheta(\blmp^t)$ of $\max_{\btheta\in\Btheta}\ipsi_\T(\btheta,\blmp^t)$.
\end{corollary}
\vspace{0.5em}

\noindent All of the maximizers for the current price $\blmp^t$ are then broadcast to the $\agent$, as in Fig. \ref{fig:block}, and the price is updated.

\subsection{Price Update}
\label{ssec:update}

The $\agent$ receives all of the maximizers from the agents, $\x(\blmp^t) = (\{\eb^i(\blmp^t)\}_{i\in\D},\{\pb^i(\blmp^t)\}_{i\in\G},\btheta(\blmp^t))$, for the current price $\blmp^t$ and uses them to compute an updated price $\blmp^{t+1}$. The price is updated in such a way as to iteratively enforce the power balance equation (see cond. (i) of Def. \ref{def:ce}). Before defining the price update, we state the following result.
\vspace{0.5em}
\begin{lemma}\label{lem:pd}
The Hessian of the Lagrangian is negative definite, that is, $\nabla^2_{\x\x}\Lb(\x,\blmp)\prec\ze \text{ for all }\x.$
\end{lemma}
\hspace{1.5em}\emph{Proof:} 
See appendix B.\hfill$\QED$
\vspace{0.5em}

\noindent As a consequence of the strong concavity of the Lagrangian, the dual function, Eq. (\ref{eq:dualfcn}), is unique and its derivative exists. The gradient of the dual function is (see Thm. 6.3.3 of \cite{bazaraa2013nonlinear})
\begin{align}
	\nabla_{\blmp} \phi(\blmp) = \h(\x(\blmp))^\top\label{eq:grad}
\end{align}
where $\h(\x)$ was defined at the beginning of Section \ref{sec:surplus}. As a result, solving the power balance equation is equivalent to finding where the gradient of the dual function vanishes. The price update follows a gradient descent method
\begin{align}
	\nonumber \blmp^{t+1} &= \blmp^t - \alpha_t\nabla_{\blmp}\phi(\blmp^t) = \blmp^t - \alpha_t\h(\x(\blmp^t))\\
	&= \blmp^t - \alpha_t\Big(\f\big(\btheta(\blmp^t)\big) - \pb(\blmp^t) + \eb(\blmp^t) + \s\Big)\label{eq:update}
\end{align}
with step-size $\alpha_t$, net generation profile $\pb(\blmp^t) = (p_1(\blmp^t),\ldots,p_N(\blmp^t))$, with $p_n(\blmp^t) = \sum_{i\in\G}p_n^i(\blmp^t)$, and net elastic demand profile $\eb(\blmp^t) = (e_1(\blmp^t),\ldots,e_N(\blmp^t))$, with $e_n(\blmp^t) = \sum_{i\in\D}e_n^i(\blmp^t)$. The injection term $\f\big(\btheta(\blmp^t)\big)$ is computed from the angle profile $\btheta(\blmp^t)$ (from Sec. \ref{ssec:transco_opt}).

\section{ANALYSIS}
\label{sec:analysis}

\subsection{Convergence Properties of the Pricing Process}
\label{ssec:conv}

We begin by showing that the dual function is Lipschitz continuous (this follows from the fact that a function with a bounded derivative is Lipschitz). By Eq. (\ref{eq:grad}), the gradient of the dual function satisfies $\nabla_{\blmp} \phi(\blmp) = \h(\x(\blmp))^\top$. Noting that $\|\h(\x(\blmp))\| = \|\f\big(\btheta(\blmp)\big) - \pb(\blmp) + \eb(\blmp) + \s \|$ and the fact that $\f\big(\btheta(\blmp)\big)$, $\pb(\blmp)$, and $\eb(\blmp)$ are all bounded, there exists some $M<\infty$ such that $\|\nabla_{\blmp} \phi(\blmp)\| = \|\h(\x(\blmp))\| \le M$. Thus the dual function is Lipschitz continuous.

The dual function $\phi(\blmp)$ is a convex function of $\blmp$. It can be shown through standard arguments that, for a sufficiently small step-size, gradient descent applied to a convex function generates iterates satisfying
\begin{align*}
	\phi(\blmp^t) - \phi^* \le \frac{\|\blmp^0 - \blmp^*\|^2 + \sum_{s=0}^t\alpha_s^2\|\nabla\phi(\blmp^s)\|^2}{2\sum_{s=0}^t\alpha_s}
\end{align*}
where $\phi^*$ denotes a minimum of $\phi$. In order to ensure convergence, one must choose $\alpha_t$ such that $\sum_{s=0}^\infty\alpha_s^2<\infty$ and $\sum_{s=0}^\infty\alpha_s=\infty$. Noting that $\|\nabla\phi(\blmp^s)\|\le M$ for all $s$, we have $\phi(\blmp^t)\to\phi^*$. Selecting a step-size of the form $\alpha_t = \beta/t$, $\beta>0$, ensures that the pricing process converges to a minimizer $\blmp^*$ of the dual function $\phi(\blmp)$, solving the dual problem (\ref{prob:dual}).\footnote{Prices must satisfy assumption 4 at each iteration $t$ in order to ensure the TransCo subproblems are convex. This can simply be achieved through choice of a sufficiently positive $\blmp^0$.} 
Since the dual problem is unconstrained, $\nabla\phi(\blmp)|_{\blmp=\blmp^*} = \ze$, and, again by Eq. (\ref{eq:grad}), $\h(\x(\blmp^*))=\ze$.

\subsection{Duality Gap}
\label{ssec:gap}

In this section, we demonstrate that the converged dual solution $\blmp^*$ results in a zero duality gap with Problem (\ref{prob:central}). 
\vspace{0.5em}
\begin{theorem}\label{thm:optimal}
The pricing process described in Section \ref{sec:solution} generates a competitive eq. $(\x^*,\blmp^*)$, where $\x^*=\x(\blmp^*)$ is a globally optimal solution to Problem (\ref{prob:central}).
\end{theorem}
\begin{proof}
Consider Problem (\ref{prob:q}), defined as 
\begin{align*}
	\max_{\bomega\in\Bomega\subseteq\Rr^W}\!\{G(\bomega)\!:\!\rb(\bomega)\!=\!(r_1(\bomega),\ldots,r_M(\bomega))\!=\!\ze\}\tag{Q}\label{prob:q}.
\end{align*}
Also, consider the following definition.
\vspace{0.25em}
\begin{definition}[Global $\bomega$-max, $\bnu$-min saddle point \cite{morgan2015explanation}] A point $(\hat \bomega,\hat \bnu)$ is a \emph{global $\bomega$-max, $\bnu$-min saddle point} for the Lagrangian $\M(\bomega,\bnu) = G(\bomega) - \bnu^\top\rb(\bomega)$ if and only if $\M(\bomega,\hat\bnu)\le\M(\hat\bomega,\hat\bnu)\le\M(\hat\bomega,\bnu)$ $\forall\,\bomega\in\Bomega$, $\bnu\in\Rr^M$.
\end{definition}
\vspace{0.25em}
The proof proceeds in two steps: (i) We prove a general result demonstrating that if $(\hat\bomega,\hat\bnu)$ is a global $\bomega$-max, $\bnu$-min saddle point for the Lagrangian $\M$ then $\hat\bomega$ is the global optimum for the problem (\ref{prob:q}) (similar to the proof found in \cite{morgan2015explanation}); (ii) We show that the pricing process generates a global $\x$-max, $\blmp$-min saddle point for the Lagrangian $\Lb$ of Problem (\ref{prob:central}).

\vspace{0.25em}
\noindent\underline{\emph{Part (i)}}: Assuming that $(\hat\bomega,\hat\bnu)$ is a global $\bomega$-max, $\bnu$-min saddle point, we have $\M(\hat\bomega,\hat\bnu)\le\M(\hat\bomega,\bnu)$ for all $\bnu$. Thus
\begin{align}\label{eq:p2_1}
	G(\hat\bomega) - \sum_{m=1}^M\hat\bnu_mr_m(\hat\bomega) \le G(\hat\bomega) -\sum_{m=1}^M\bnu_mr_m(\hat\bomega).
\end{align}
Let there exist an index $m'$ such that $r_{m'}(\hat\bomega) >0$, then we can choose $\nu_{m'}\gg0$ such that Eq. (\ref{eq:p2_1}) is violated. Similarly, let there exist an index $m''$ such that $r_{m''}(\hat\bomega) <0$, we can violate Eq. (\ref{eq:p2_1}) by choosing $\nu_{m''}\ll0$. Thus $r_{m}(\hat\bomega)=0$ for all $m$ and therefore $\hat\bomega$ is feasible for Problem (\ref{prob:q}).

Since $(\hat\bomega,\hat\bnu)$ is a global $\bomega$-max, $\bnu$-min saddle point, we also have $\M(\bomega,\hat\bnu)\le\M(\hat\bomega,\hat\bnu)$ for all $\bomega\in\Bomega$. Thus $G(\bomega) - \sum_{m=1}^M\hat\bnu_mr_m(\bomega) \le G(\hat\bomega) - \sum_{m=1}^M\hat\bnu_mr_m(\hat\bomega)$. Since $\rb(\bomega)=\ze$ for every feasible $\bomega$, we have that $G(\bomega)\le G(\hat\bomega)$ everywhere on $\{\bomega|\rb(\bomega)=\ze,\bomega\in\Bomega\}$ and thus $\hat\bomega$ is optimal for Problem (\ref{prob:q}).

\vspace{0.25em}
\noindent\underline{\emph{Part (ii)}}: Let $\x^* = \x(\blmp^*)= \argmax_{\x\in\X}\Lb(\x,\blmp^*)$, where $\blmp^*$ is the converged price vector obtained from the pricing process (Eq. \ref{eq:update}). The profile $\x^*$ is returned from the agents when provided with the price $\blmp^*$ (via the optimizations in Section \ref{ssec:response}). By assumption 3 and Corollary \ref{cor:admm}, $\x^*$ is unique. Notice that $(\x^*,\blmp^*)$ is a competitive equilibrium by Def. \ref{def:ce}. Also, due to the concavity of $\Lb$ (Lemma \ref{lem:pd}), $\phi(\blmp^*) = \Lb(\x^*,\blmp^*)\ge\Lb(\x,\blmp^*)$ for all $\x\in\X$. We know that $\nabla\phi(\blmp)|_{\blmp=\blmp^*} = \ze$ and thus, by Eq. (\ref{eq:grad}), $\h(\x(\blmp^*))=\h(\x^*)=\ze$. Consequently, $\Lb(\x^*,\blmp^*) = J(\x^*) -(\blmp^*)^\top\h(\x^*) =J(\x^*)= J(\x^*)-\blmp^\top\h(\x^*) = \Lb(\x^*,\blmp)$ 
for all $\blmp$. In summary $\Lb(\x,\blmp^*)\le\Lb(\x^*,\blmp^*)=\Lb(\x^*,\blmp)$ for all $\x\in\X,\,\blmp$ and thus $(\x^*,\blmp^*)$ is a global $\x$-max, $\blmp$-min saddle point for the Lagrangian $\Lb(\x,\blmp)$.

From parts (i) and (ii), we conclude that the pricing process generates the pair $(\x^*,\blmp^*)$ where $\x^*$ is a globally optimal solution to Problem (\ref{prob:central}).
\end{proof}

In summary, the pricing process generates the competitive equilibrium $(\x^*,\blmp^*)$ which results in a globally optimal solution $\x^*$ to the (nonconvex) social welfare maximization problem (\ref{prob:central}). Consequently, $\x^*$ is a Pareto efficient outcome. 

\vspace{0.5em}
\section{NUMERICAL EXAMPLE}

We demonstrate the performance of the pricing process on a modified version of the IEEE 14 bus test system. The ownership of generators in the modified system is split among three GenCos with $\pb^1 = (p_1^1,p_2^1)$, $\pb^2 = (p_3^2,p_6^2)$, and $\pb^3=(p_8^3)$. The network also consists of seven DistCos with $\eb^1 = (e_2^1,e_3^1)$, $\eb^2=(e_3^2,e_4^2)$, $\eb^3=(e_5^3)$, $\eb^4=(e_6^4,e_{11}^4,e_{12}^4)$, $\eb^5=(e_9^5,e_{10}^5)$, $\eb^6=(e_{12}^6,e_{13}^6)$, $\eb^7=(e_{14}^7)$ and inelastic demands (in MW) $s_2^1 = 15$, $s_5^3 = 10$, $s_{12}^4=15$, $s_{10}^5=10$, $s_{14}^6=15$. The ownership of lines is split among two TransCos, 
{
    \def\OldComma{,}
    \catcode`\,=13
    \def,{%
      \ifmmode%
        \OldComma\discretionary{}{}{}%
      \else%
        \OldComma%
      \fi%
    }%
$\E^1 = \big\{\!\{1,\!2\},\!\{1,\!5\},\!\{2,\!3\},\!\{2,\!4\},\!\{2,\!5\},\!\{3,\!4\},\!\{4,\!5\},\!\{4,\!7\},\!\{5,\!6\}\!\big\}$, $\E^2 = \big\{\!\{4,\!9\},\!\{6,\!11\},\!\{6,\!12\},\!\{6,\!13\},\!\{7,\!8\},\!\{7,\!9\},\!\{9,\!10\},\!\{9,\!14\},\!\{10,\!11\},\!\{12,\!13\},\!\{13,\!14\}\!\big\}$ with slack bus $\Ss = \{6\}$. 
}
Parameters for the TransCo message exchange process (Alg. 1) are $\rho = 0.21$, $\varepsilon_{\text{primal}} = 5$\ex{5}, $\varepsilon_{\text{dual}}=5$\ex{6}. Fig. \ref{fig:results} demonstrates the convergence of the process.

\begin{figure}[h!]
\begin{center}
\includegraphics[width=0.95\columnwidth]{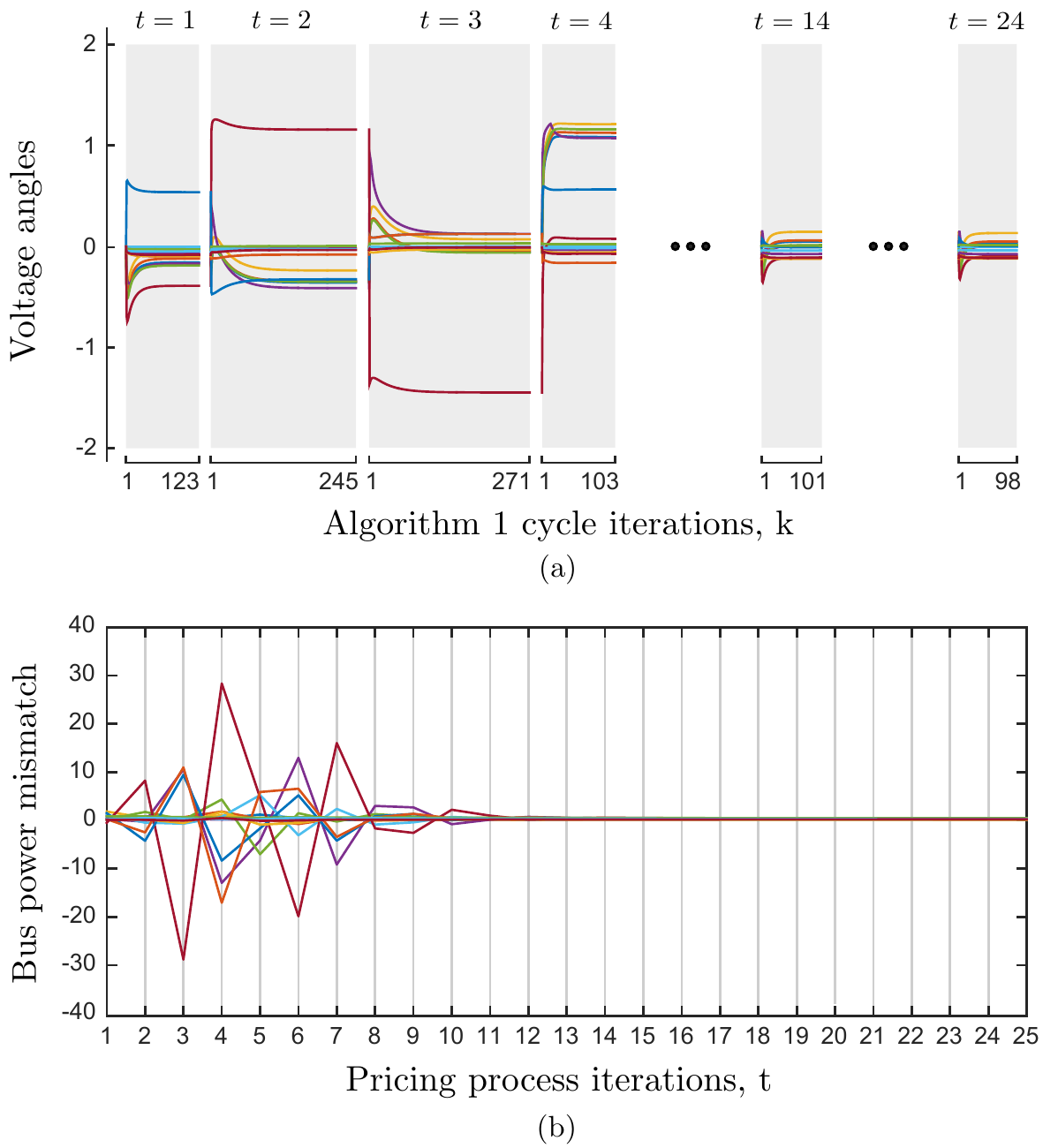}
\caption{\emph{Convergence of pricing process:} (a) TransCos reach an angle agreement for each price vector $\blmp^t$ via Alg. 1 (each negotiation cycle corresponds to a price vector); (b) The power mismatch at each bus $n$, $h_n(\x(\blmp^t)) = f_n(\btheta(\blmp^t)) - p_n(\blmp^t) + e_n(\blmp^t) + s_n$, converges to zero.}\label{fig:results}
\end{center}
\vspace{-2em}
\end{figure}

It is evident from Fig. \ref{fig:results}(b) that the pricing process generates a solution where the power balance equation is satisfied (condition (i) of Def. \ref{def:ce}). At the corresponding prices, agents report their surplus-maximizing responses, satisfying condition (ii) of Def. \ref{def:ce}. By Theorem \ref{thm:optimal}, the resulting competitive equilibrium is Pareto efficient.

\vspace{0.5em}
\section{CONCLUSIONS \& FUTURE WORK}
\label{sec:conc}

We have developed an electricity market model and an associated decentralized market mechanism that, under natural assumptions, ensures convergence to a Pareto efficient market (competitive) equilibrium. The market model includes multiple DistCos, GenCos, and (cooperative) TransCos all of which are assumed to be surplus-maximizing given the current set of LMPs. A market operator updates LMPs via a gradient method in order to achieve an operating point that satisfies the power balance equations and consequently clears the market.


\section*{APPENDIX}

\subsection{Proof of Lemma \ref{lem:ms}}

Let $\vb=\btheta^i$. The Hessian of $\ipsi_\T^i(\vb,\blmp)$, see Eq. (\ref{eq:ms}), with respect to $\vb$ is $\nabla^2_{\vb\vb}\ipsi_\T^i\big(\vb,\blmp\big)=-\frac{1}{2}\nabla^2_{\vb\vb}\big(\sum_{(n,m)\in\Ed^i}(\lambda_n + \lambda_m)\tilde g(v_{nm})\big)$, where the first order terms do not enter into the expression. Define $\iota=\N_\T^i$ as the (ordered) set of bus indices of TransCo $i$ and define
\begin{align*}
	A_{jk} =
	 \left\{\begin{array}{ll}   
	-\!\!\!\!\!\displaystyle\sum_{\{j,l\}\in\E^i}\!\!(\lambda_{\iota_j} + \lambda_{\iota_l})G_{\iota_j\iota_l} &  \text{ if } j=k\\
	-(\lambda_{\iota_j} + \lambda_{\iota_k})G_{\iota_j\iota_k} &  \text{ if } \{j,k\}\in\E^i\\[0.25em]
	0 &  \text{ if } \{j,k\}\not\in\E^i
\end{array}\right.
\end{align*}
for each $j,k=1,\ldots,|\iota| = |\N_\T^i|$. By assumption 2, there exists an index $s\in\iota$ that corresponds to a slack bus. The Hessian $\nabla^2\ipsi_\T^i$ is defined as matrix $A$ with the $s^\text{th}$ row and column removed. Consequently, $\nabla^2\ipsi_\T^i$ belongs to the class of \emph{irreducibly diagonally dominant} matrices, known to be non-singular (see Theorem 6.2.27 of \cite{horn1985matrix}). To see this, note that the Hessian is diagonally dominant for all rows. Additionally, it is strictly diagonally dominant in rows that correspond to buses that are immediately connected to a slack bus. By assumption 4, the diagonal elements of the Hessian are negative and by Prop. 2.2.20 of \cite{cottle1992linear}, $\nabla^2\ipsi_\T^i\prec\ze$.

\subsection{Proof of Lemma \ref{lem:pd}}

The Hessian of the Lagrangian, denoted by $\nabla_{\x\x}^2\Lb$, is a square, block-diagonal matrix of dimension $\sum_{i\in\D}|\N_{\D^e}^i| + \sum_{i\in\G}|\N_\G^i| + |\N\setminus\Ss|$. It consists of three blocks, ${\bf U}$, ${\bf C}$, and ${\bf M}$, where ${\bf U}$ and ${\bf C}$ are diagonal matrices consisting of elements $(u_n^i)''$ (corresponding to DistCo units) and $-(c_n^i)''$ (corresponding to GenCo units), respectively. By assumption 3, we have ${\bf U},{\bf C}\prec\ze$. Similar to the proof of Lemma \ref{lem:ms}, matrix ${\bf M}$ can be shown to be an irreducibly diagonally dominant matrix with a negative diagonal and thus, again by Prop. 2.2.20 of \cite{cottle1992linear}, ${\bf M}\prec\ze$. Since $\nabla_{\x\x}^2\Lb$ is block-diagonal with each block negative definite, we conclude $\nabla_{\x\x}^2\Lb\prec\ze$.

\bibliographystyle{IEEEtran}
\bibliography{references}

\end{document}